\documentclass[ECP, preprint]{ejpecp} 

\usepackage[T1]{fontenc}
\usepackage[utf8]{inputenc}
\usepackage{latexsym,amssymb,amsfonts}
\usepackage{mathtools}
\usepackage{enumitem}

\newcommand{\field}[1]{\mathbb{#1}}
\newcommand{\R}{\field{R}}

\newcommand{\1}{\mathbb{1}}

\newcommand{\Exp}{\operatorname{\mathbb E}}

\newcommand{\Tr}{\operatorname{Tr}}

\newcommand{\norm}[1]{\lVert #1 \rVert}

\DeclareMathOperator{\osc}{osc}

\providecommand{\orcid}[1]{}

\SHORTTITLE{Fluctuations in Averaging-Learning}

\TITLE{Rank-One Fluctuations in Averaging-Learning Dynamics}

\AUTHORS{
  Ionel Popescu\footnote{University of Bucharest.
    \EMAIL{ionel.popescu@fmi.unibuc.ro}}\orcid{0000-0001-5150-5568}
  \and
  Tushar Vaidya\footnote{Nanyang Technological University.
    \EMAIL{tushar.vaidya@ntu.edu.sg}}\orcid{0000-0002-2264-2595}
}


\AMSSUBJ{60F05; 39A30; 15B51}

\SUBMITTED{}
\ACCEPTED{}

\VOLUME{0}
\YEAR{2026}
\PAPERNUM{0}
\DOI{10.1214/YY-TN}

\ABSTRACT{We study averaging-learning dynamics without an exogenous ground truth:
the reference signal is generated endogenously by the population.  The dynamics
combine a time-varying averaging matrix, a learning-source matrix, and a learning
matrix $\mathcal{E}_t$, typically diagonal.  Dobrushin-type contraction controls
the decay of oscillations and yields asymptotic agreement.  Under a summability
condition the backward products $B_{t:s}$ converge exponentially to rank-one
limits $F_s=\1\nu_s$. With summable perturbations, the process converges to a random consensus state.  For i.i.d. perturbations we prove a central
limit theorem for the centered process: the limiting Gaussian law is supported
on the agreement direction.  Thus, despite the multi-agent dynamics, the
long-time fluctuations are asymptotically one-dimensional. We also record a pairwise Dobrushin formulation which clarifies the dynamic agreement-class geometry underlying the one-class regime.}

\begin{document}

\section{Introduction}

Averaging-plus-learning models update the opinion of agent $i$ by
\begin{equation}
x^i_{t+1}
=
\sum_{j=1}^n a_{ij}(t)x^j_t
+
\varepsilon_{i,t}(\bar\sigma-x^i_t),
\end{equation}
where $A_t=(a_{ij}(t))$ is row-stochastic and $\bar\sigma$ is an exogenous truth
signal.  This is the averaging-plus-learning setup of
\cite{popescu2023averaging, popescu2026mixednorms}.  In the present paper, the reference signal is not
exogenous.  It is produced endogenously by the population.  The work here sits
within the broad class of repeated averaging and opinion dynamics models
originating with DeGroot \cite{degroot1974reaching}; see also \cite{ChatterjeeSeneta1977Consensus} for classical convergence
results and \cite{AAAP2025} for robust naive learning. Let
\[
        X_t=(x^1_t,\ldots,x^n_t)^\top\in\R^n .
\]
The matrix $A_t$ describes averaging among agents.  The row-stochastic matrix
$\Delta_t$ describes the agents from whom each player learns.  The learning
intensity matrix is denoted throughout by
\[
        \mathcal{E}_t
        =
        \operatorname{diag}(\varepsilon_{1,t},\ldots,\varepsilon_{n,t}),
\]
or, more generally, by a deterministic matrix with uniformly bounded
$\ell^\infty$ norm.  In the full-information case one may take
\[
        \Delta_t=n^{-1}J,
        \qquad
        J=\1\1^\top \text{ with }\1=(1,1,\dots, 1)^\top.
\]
We study the recursion
\begin{equation}\label{eq:AVM2}
X_t
=
A_tX_{t-1}
+
\mathcal{E}_t\bigl(\Delta_t X_{t-1}-X_{t-1}+\gamma_t\bigr),
\qquad t\ge1,
\end{equation}
where $\gamma_t\in\R^n$ is the \emph{perturbation}.  The central object of the paper is the matrix
\begin{equation}\label{eq:def-Bt}
  B_t:=A_t+\mathcal{E}_t(\Delta_t-I),
\end{equation}
 and the recursion becomes
\begin{equation}\label{eq:linear-recursion}
        X_t=B_tX_{t-1}+\mathcal{E}_t\gamma_t .
\end{equation}

Since $A_t\1=\1$ and $\Delta_t\1=\1$, each $B_t$ satisfies
\[
        B_t\1=\1 .
\]
Thus the agreement direction is invariant.  We use the word \emph{agreement}
for the decay of differences among agents, allowing the common value to move in
time.  We reserve \emph{consensus} for convergence to a time-independent common
value.  Throughout the paper the matrices are deterministic; randomness enters only through the perturbations. The basic mechanism is Dobrushin contraction of oscillations.  The oscillation
\[
        \osc(x)=\max_{i,j}|x_i-x_j|
\]
measures disagreement, and the Dobrushin coefficient of a matrix with equal row sums measures how strongly the matrix contracts this disagreement.  Decay of the corresponding backward-product contraction factors yields agreement for the homogeneous recursion. Under a summability condition, the backward products have exponential rank-one limits: denoting $B_{t:s}=B_tB_{t-1}\dots B_s$ we write
\[
        B_{t:s}\xrightarrow[t\to\infty]{} F_s:=\1\nu_s .
\] 
With summable perturbations, the process converges to a random consensus state. The main probabilistic result is a central limit theorem for the centered process under i.i.d. perturbations.  Disagreement is negligible at scale $\sqrt t$.  What remains is a fluctuation in the agreement direction, and the limiting Gaussian law is rank one.  The mechanism is a single scalar projection which removes the disagreement directions and leaves one diffusive coordinate.

The same estimates also have a finer interpretation.  Instead of taking the
maximum over all row pairs in the Dobrushin coefficient, one may keep the
individual row-pair distances of the backward products.  This pairwise geometry
gives a quantitative way to define dynamically generated agreement classes.
Under the hypotheses of the present CLT all these classes merge into one,
producing the rank-one covariance; beyond the global maximum condition one is
led naturally to a block-consensus picture.

Dobrushin's theorem and the martingale proof of Sethuraman and Varadhan
\cite{Dobrushin1956NonstationaryI,Dobrushin1956NonstationaryII,
SethuramanVaradhan2005Dobrushin} are important precursors for contraction-based
Gaussian limits in time-inhomogeneous systems. The present result is different
in both setting and conclusion. It is formulated for endogenous
averaging-learning dynamics, and the limiting fluctuations collapse onto the agreement direction. Unlike the non-homogeneous Markov-chain CLT, the randomness here is injected into an endogenous averaging-learning recursion, while Dobrushin contraction acts on
the agents' state (belief) vector by eliminating transverse disagreement directions. What remains at diffusive scale is a single time-dependent agreement coordinate. We first isolate the Dobrushin estimates that control disagreement throughout the paper.

\section{Oscillation and Dobrushin contraction}

For $x\in\R^n$ define
\[
        \osc(x):=\max_{1\le i,j\le n}|x_i-x_j|.
\]
Let
\[
        H:=\{x\in\R^n:\1^\top x=0\}
\]
be the disagreement subspace.  Every $x\in\R^n$ has the decomposition
\begin{equation}
        x=\bar x\,\1+x^\perp,
        \qquad
        \bar x:=\frac1n\1^\top x,
        \qquad
        x^\perp:=x-\bar x\,\1\in H.
\end{equation}
For a matrix $M=(m_{ij})$ with equal row sums, its Dobrushin coefficient is
\begin{equation}
        \delta(M)
        :=
        \frac12\max_{1\le i,j\le n}
        \sum_{k=1}^n |m_{ik}-m_{jk}|.
\end{equation}
All matrix norms are induced by $\ell^\infty$, so
\[
        \norm{M}_\infty=\max_i\sum_j |m_{ij}|,
\]
unless another norm is explicitly displayed.

\begin{lemma}\label{lem:dobrushin}
Let $x\in H$ and let $M,N$ be matrices whose row sums are constant.
\begin{enumerate}[label=(\roman*)]
\item $\norm{x}_\infty\le \osc(x)\le2\norm{x}_\infty$.
\item $\osc(My)\le\delta(M)\osc(y)$ for every $y\in\R^n$.
\item $\delta(MN)\le\delta(M)\delta(N)$.
\end{enumerate}
If, in addition, $M$ is stochastic, then $\delta(M)\le1$.
\end{lemma}

\begin{proof}
For (i), reorder the coordinates so that $x_1\le\cdots\le x_n$.  Since
$x\in H$, $x_1\le0\le x_n$, hence
\[
        \norm{x}_\infty
        =
        \max\{-x_1,x_n\}
        \le
        x_n-x_1
        =
        \osc(x)
        \le
        2\norm{x}_\infty .
\]

For (ii), write
\[
        y_k=\frac{y_{\min}+y_{\max}}2+z_k,
        \qquad
        |z_k|\le\frac12\osc(y).
\]
Equal row sums cancel the constant part, so
\[
        (My)_i-(My)_j
        =
        \sum_{k=1}^n(m_{ik}-m_{jk})z_k .
\]
Taking absolute values and maximizing over $i,j$ yields $\osc(My)\le\delta(M)\osc(y)$. Moreover, for any matrix $P$ with constant row sums,
\[
        \delta(P)
        =
        \sup_{\osc(y)\le1}\osc(Py).
\]
Indeed, the inequality ``$\le$'' follows from the preceding argument, while the
reverse inequality is obtained by fixing a row pair $i,j$ and taking
$y_k=\frac12\operatorname{sign}(p_{ik}-p_{jk})$.
Applying this twice gives (iii):
\[
        \osc(MNy)
        \le
        \delta(M)\osc(Ny)
        \le
        \delta(M)\delta(N)\osc(y).
\]
If $M$ is stochastic, its rows are probability vectors, and their total
variation distance is at most one.
\end{proof}
For integers $t\ge s$ write
\begin{equation}\label{eq:products}
        D_{t,s}:=\prod_{u=s}^t\delta(B_u),
        \qquad
        B_{t:s}:=B_tB_{t-1}\cdots B_s,
\end{equation}
with $B_{t:s}=I$ and $D_{t,s}=1$ when $t<s$.  The main summability hypothesis
used below is
\begin{equation}\label{eq:mainhyp}
        A_*:=\sup_{t\ge1}\sum_{k=1}^tD_{t,k}<\infty,
        \qquad
        L:=\sup_{t\ge1}\norm{B_t}_\infty<\infty,
        \qquad
        K:=\sup_{t\ge1}\norm{\mathcal{E}_t}_\infty<\infty .
\end{equation}
By \cite[Lemma 17]{popescu2023averaging}, the first condition in
\eqref{eq:mainhyp} implies exponential decay: there are constants $c,C>0$,
depending only on $A_*$, such that
\begin{equation}\label{eq:expdecay}
        D_{t,s+1}\le Ce^{-c(t-s)},
        \qquad t\ge s\ge0 .
\end{equation}
Since $\delta(B_s)\le\norm{B_s}_\infty\le L$, the same estimate, with a
different constant $C$, also controls $D_{t,s}$ for $t\ge s$.
\section{Agreement, consensus, and backward-product limits}\label{sec:Agreementconsenus}
We first treat the deterministic recursion
\[
        X_t=B_tX_{t-1}.
\]

\begin{theorem}[Consensus convergence]\label{thm:deterministic}
Assume $\gamma_t\equiv0$.
\begin{enumerate}[label=(\arabic*)]
\item If $D_{t,1}\to0$, then $\osc(X_t)\to0$.
\item If either
\begin{enumerate}[label=(\alph*)]
\item each $B_t$ is stochastic and $D_{t,1}\to0$, or
\item $\sup_t\norm{B_t}_\infty<\infty$ and
      $\sum_{t\ge1}D_{t,1}<\infty$,
\end{enumerate}
then $X_t\to\lambda\1$ for some $\lambda\in\R$.
\end{enumerate}
\end{theorem}

\begin{proof}
\begin{enumerate}
\item By Lemma~\ref{lem:dobrushin}(ii),
\begin{equation}\label{eq:det-osc-bound}
        \osc(X_t)\le D_{t,1}\osc(X_0),
\end{equation}
which proves the first assertion.

\item If every $B_t$ is stochastic, then $X_t=B_tX_{t-1}$ is obtained by taking convex
combinations of the coordinates of $X_{t-1}$, hence
\[
        \min_i(X_{t-1})_i
        \le
        \min_i(X_t)_i
        \le
        \max_i(X_t)_i
        \le
        \max_i(X_{t-1})_i .
\]
The minima increase, the maxima decrease, and their gap is
$\osc(X_t)\to0$.  The two limits coincide, so $X_t\to\lambda\1$.

Now assume (b).  Decompose
\[
        X_t=m_t\1+R_t,
        \qquad
        R_t\in H .
\]
Since $B_t\1=\1$,
\begin{equation}
        (m_t-m_{t-1})\1=B_tR_{t-1}-R_t .
\end{equation}
Using Lemma~\ref{lem:dobrushin}(i) and \eqref{eq:det-osc-bound},
\[
\begin{aligned}
        |m_t-m_{t-1}|
        &\le
        L\norm{R_{t-1}}_\infty+\norm{R_t}_\infty      \\
        &\le
        L\osc(X_{t-1})+\osc(X_t)                      \\
        &\le
        C(D_{t-1,1}+D_{t,1}) .
\end{aligned}
\]
The right-hand side is summable.  Hence $(m_t)$ is Cauchy, while
$\norm{R_t}_\infty\le\osc(X_t)\to0$.  Therefore $X_t\to\lambda\1$.\qedhere
\end{enumerate}
\end{proof}

The next proposition identifies the limiting backward products.  We denote these
limits by $F_s$.

\begin{proposition}[Rank-one backward-product limit]\label{prop:rankone}
Assume \eqref{eq:mainhyp}.  For each $s\ge1$ there is a row vector $\nu_s$ with
$\nu_s\1=1$ such that, with
\[
        F_s:=\1\nu_s,
\]
one has
\begin{equation}\label{eq:rank-one-limit}
        \norm{B_{t:s}-F_s}_\infty
        \le
        Ce^{-c(t-s)},
        \qquad t\ge s .
\end{equation}
Moreover,
\begin{equation}\label{eq:nu-consistency}
        F_s=F_{s+1}B_s,
        \qquad
        \nu_s=\nu_{s+1}B_s,
        \qquad s\ge1 .
\end{equation}
\end{proposition}

\begin{proof}
Fix $s$ and a basis vector $e_j$.  Let
\[
        Y_t=B_{t:s}e_j,
        \qquad t\ge s .
\]
By \eqref{eq:expdecay} and Lemma~\ref{lem:dobrushin},
\begin{equation}
        \osc(Y_t)\le Ce^{-c(t-s)} .
\end{equation}
Write $Y_t=m_t\1+R_t$, with $R_t\in H$.  The shifted version of
Theorem~\ref{thm:deterministic}(2b) applies because
\[
        \sum_{t\ge s}D_{t,s}<\infty,
        \qquad
        \sup_t\norm{B_t}_\infty<\infty .
\]
Thus $Y_t$ converges to a consensus vector, say $\lambda_{j,s}\1$.

The exponential rate follows from the scalar component as well as from the
oscillation.  Using $Y_{t+1}=B_{t+1}Y_t$ and $B_{t+1}\1=\1$,
\[
        |m_{t+1}-m_t|
        \le
        L\norm{R_t}_\infty+\norm{R_{t+1}}_\infty
        \le
        Ce^{-c(t-s)} .
\]
Summing the tail gives
\[
        |m_t-\lambda_{j,s}|\le Ce^{-c(t-s)},
        \qquad
        \norm{R_t}_\infty\le Ce^{-c(t-s)} .
\]
Therefore
\[
        \norm{B_{t:s}e_j-\lambda_{j,s}\1}_\infty
        \le
        Ce^{-c(t-s)} .
\]
Collecting the limits over $j$ gives a matrix $F_s$ with equal rows.  Write
$F_s=\1\nu_s$.  Since each $B_{t:s}$ has row sums one, also $F_s\1=\1$, hence
$\nu_s\1=1$.  Summing over the columns yields \eqref{eq:rank-one-limit}, after
changing $C$ by a dimension-dependent factor.

Finally,
\[
        F_s
        =
        \lim_{t\to\infty}B_{t:s}
        =
        \lim_{t\to\infty}B_{t:s+1}B_s
        =
        F_{s+1}B_s
        =
        \1(\nu_{s+1}B_s).
\]
Thus $\nu_s=\nu_{s+1}B_s$.\qedhere
\end{proof}

Such rank-one limits of backward products belong to the classical theory of
stochastic-matrix products; see, for example, \cite{Stenflo2008PerfectSampling}.

\section{Summable perturbations}\label{sec:summperturb}

We return to the full recursion \eqref{eq:linear-recursion}.  The oscillation
satisfies an exponentially weighted convolution estimate.

\begin{lemma}\label{lem:oscconv}
Under \eqref{eq:mainhyp}, for every $t\ge1$,
\begin{equation}\label{eq:osc-conv}
        \osc(X_t)
        \le
        Ce^{-ct}\osc(X_0)
        +
        2KC\sum_{s=1}^t e^{-c(t-s)}\norm{\gamma_s}_\infty .
\end{equation}
Consequently,
\begin{equation}\label{eq:osc-sum}
        \sum_{t\ge1}\osc(X_t)
        \le
        C\left(\osc(X_0)+\sum_{t\ge1}\norm{\gamma_t}_\infty\right)
\end{equation}
pointwise.  If the right-hand side belongs to $L^p$, then so does the left-hand side. Moreover, if $\Exp\osc(X_0)<\infty$, and $(\gamma_t)$ are identically distributed with $\Exp\norm{\gamma_1}_\infty<\infty$, then
\[
        \sup_{t\ge1}\Exp\osc(X_t)<\infty.
\]
\end{lemma}

\begin{proof}
From \eqref{eq:linear-recursion} and Lemma~\ref{lem:dobrushin},
\[
        \osc(X_t)
        \le
        \delta(B_t)\osc(X_{t-1})
        +
        \osc(\mathcal{E}_t\gamma_t)
        \le
        \delta(B_t)\osc(X_{t-1})
        +
        2K\norm{\gamma_t}_\infty .
\]
Iteration and \eqref{eq:expdecay} give \eqref{eq:osc-conv}.  Summing
\eqref{eq:osc-conv} over $t$ and exchanging the order of summation gives
\eqref{eq:osc-sum}, since
\[
        \sum_{m\ge0}e^{-cm}<\infty .
\]
The $L^p$ assertion follows by monotonicity and the same deterministic bound.
The final assertion follows directly from \eqref{eq:osc-conv} and stationarity
of the one-dimensional expectations.\qedhere
\end{proof}

\begin{theorem}[Consensus with summable perturbations]\label{thm:summable}
Assume \eqref{eq:mainhyp} and $1 \le p \le \infty$.
\begin{enumerate}[label=(\arabic*)]
\item If $\sum_{t\ge1}\norm{\gamma_t}_\infty<\infty$ almost surely, then
      $X_t\to Z\1$ almost surely for some random variable $Z$.
\item If $X_0\in L^p$ and $\sum_{t\ge1}\norm{\gamma_t}_\infty$ converges in $L^p$, then
      $X_t\to Z\1$ in $L^p$ for some $Z\in L^p$.
\end{enumerate}
\end{theorem}

\begin{proof}
Write
\[
        X_t=m_t\1+R_t,
        \qquad
        R_t\in H .
\]
Since $B_t\1=\1$,
\begin{equation}
        (m_t-m_{t-1})\1
        =
        B_tR_{t-1}-R_t+\mathcal{E}_t\gamma_t .
\end{equation}
Hence
\begin{equation}\label{eq:mt-increment-bound}
\begin{aligned}
        |m_t-m_{t-1}|
        &\le
        L\norm{R_{t-1}}_\infty+\norm{R_t}_\infty
        +
        K\norm{\gamma_t}_\infty                         \\
        &\le
        L\osc(X_{t-1})+\osc(X_t)
        +
        K\norm{\gamma_t}_\infty .
\end{aligned}
\end{equation}
If $\sum_t\norm{\gamma_t}_\infty<\infty$ almost surely,
Lemma~\ref{lem:oscconv} gives
\[
        \sum_t\osc(X_t)<\infty
\]
almost surely.  Thus the right-hand side of \eqref{eq:mt-increment-bound} is
summable, $(m_t)$ is almost surely Cauchy, and
\[
        \norm{R_t}_\infty\le\osc(X_t)\to0 .
\]
This proves $X_t\to Z\1$ almost surely.

The $L^p$ proof is the same after taking $L^p$ norms in \eqref{eq:osc-sum} and
\eqref{eq:mt-increment-bound}: $\sum_t |m_t-m_{t-1}|$ converges in $L^p$, so
$(m_t)$ is Cauchy in $L^p$, and $R_t\to0$ in $L^p$.
\end{proof}

\section{Central limit theorem}

We give two proofs. The first keeps the full covariance matrix visible and is stable under vector-valued extensions; the second identifies the scalar coordinate responsible for the limiting fluctuation.

Assume now that $(\gamma_t)_{t\ge1}$ are i.i.d. with mean $\mu\in\R^n$ and
covariance matrix $\Sigma$, and that $X_0$ is deterministic.  For nonsummable i.i.d. perturbations, Theorem~\ref{thm:summable} no longer gives convergence. We therefore study the centered process at diffusive scale. Put $\xi_t:=\gamma_t-\mu$.

\begin{lemma}[Ces\`aro square control]\label{lem:cesaro-square-control}
Let $(a_s)_{s\ge1}$ be a nonnegative sequence.  If
\[
        \frac1t\sum_{s=1}^t a_s^2 \to A<\infty,
\]
then
\[
        \frac{a_t}{\sqrt t}\to0,
        \qquad
        \max_{s\le t}\frac{a_s}{\sqrt t}\to0 .
\]
In particular, if
\[
        \frac1t\sum_{s=1}^t\nu_s^\top\nu_s\to\mathcal C,
\]
then the Ces\`aro averages of $\norm{\nu_s}_2^2$ are bounded and
\[
        \max_{s\le t}\frac{\norm{\nu_s}_2}{\sqrt t}\to0 .
\]
The same conclusions hold with $t$ replaced by $t+1$ in the upper index.
\end{lemma}

\begin{proof}
Put $S_t:=\sum_{s=1}^t a_s^2$.  Since $S_t/t\to A$,
$
        \frac{a_t^2}{t}
        =
        \frac{S_t-S_{t-1}}{t}
        =
        \frac{S_t}{t}
        -
        \frac{t-1}{t}\frac{S_{t-1}}{t-1}
        \to0 .$
For the maximum, fix $\varepsilon>0$.  Choose $N$ so that
$a_s/\sqrt s\le\varepsilon$ for all $s\ge N$.  Then, for $t\ge N$,
\[
        \max_{s\le t}\frac{a_s^2}{t}
        \le
        \max\left\{
        \frac{\max_{s<N}a_s^2}{t},
        \varepsilon^2
        \right\} .
\]
Letting $t\to\infty$ and then $\varepsilon\downarrow0$ gives the second claim.
The final assertion follows by taking traces, because
\[
        \frac1t\sum_{s=1}^t\norm{\nu_s}_2^2
        =
        \Tr\!\left(\frac1t\sum_{s=1}^t\nu_s^\top\nu_s\right)
        \to\Tr(\mathcal C),
\]
and by the equivalence of finite-dimensional norms.
\end{proof}

\begin{theorem}[Fluctuation theorem]\label{thm:clt}
Assume \eqref{eq:mainhyp} and suppose that
\begin{equation}\label{eq:clt-hyp}
        \mathcal{E}_t\to\mathcal{E}_\infty,
        \qquad
        \frac1t\sum_{s=1}^t\nu_s^\top\nu_s\to \mathcal{C} .
\end{equation}
Then
\begin{equation}\label{eq:clt-limit}
        \frac{X_t-\Exp[X_t]}{\sqrt t}
        \Rightarrow
        N(0,\sigma^2J),
        \qquad
        \sigma^2
        :=
        \Tr\!\left(\mathcal{C}\mathcal{E}_\infty\Sigma\mathcal{E}_\infty^\top\right), \text{ and }J=\1\1^\top.
\end{equation}
The limiting Gaussian law is supported on the line $\R\1$.
\end{theorem}

\begin{proof}[First proof: direct matrix proof]
From \eqref{eq:linear-recursion},
\begin{equation}
        X_t-\Exp[X_t]
        =
        \sum_{s=1}^t B_{t:s+1}\mathcal{E}_s\xi_s .
\end{equation}
Set $M_{t,s}:=B_{t:s+1}\mathcal{E}_s$. The covariance of the right-hand side is
\begin{equation}
        \Gamma_t
        :=
        \sum_{s=1}^t M_{t,s}\Sigma M_{t,s}^\top .
\end{equation}
We first prove that $t^{-1}\Gamma_t\to\sigma^2J$. The terminal term $s=t$ contributes only
$\mathcal{E}_t\Sigma\mathcal{E}_t^\top$ to $\Gamma_t$, which is $O(1)$ in $\ell^\infty$ norm and is negligible after division by $t$.  For $s<t$,
Proposition~\ref{prop:rankone} gives
\begin{equation}
        B_{t:s+1}=F_{s+1}+W_{t,s},
        \qquad
        F_{s+1}=\1\nu_{s+1},
        \qquad
        \norm{W_{t,s}}_\infty\le Ce^{-c(t-s)} .
\end{equation}
Thus
\begin{equation}
        \Gamma_t
        =
        \mathcal{E}_t\Sigma\mathcal{E}_t^\top
        +
        V_t+R_t^{(1)}+R_t^{(2)}+R_t^{(3)},
\end{equation}
where the four sums below run over $1\le s\le t-1$:
\begin{equation}
\begin{aligned}
V_t
&:=
\sum_{s=1}^{t-1}
F_{s+1}\mathcal{E}_s\Sigma\mathcal{E}_s^\top F_{s+1}^\top,
& R_t^{(1)}:=
\sum_{s=1}^{t-1}
W_{t,s}\mathcal{E}_s\Sigma\mathcal{E}_s^\top F_{s+1}^\top,
\\
R_t^{(2)}
&:=
\sum_{s=1}^{t-1}
F_{s+1}\mathcal{E}_s\Sigma\mathcal{E}_s^\top W_{t,s}^\top,
& R_t^{(3)}:=
\sum_{s=1}^{t-1}
W_{t,s}\mathcal{E}_s\Sigma\mathcal{E}_s^\top W_{t,s}^\top .
\end{aligned}
\end{equation}
The last term is bounded uniformly in $t$:
\[
        \norm{R_t^{(3)}}_\infty
        \le
        C\sum_{s=1}^{t-1}e^{-2c(t-s)}
        \le
        C .
\]
For the cross terms, using
\[
        \norm{F_{s+1}}_\infty\le n\norm{\nu_{s+1}}_\infty
        \qquad\text{and}\qquad
        \norm{M^\top}_\infty\le n\norm{M}_\infty,
\]
we obtain
\[
        \norm{R_t^{(1)}}_\infty+\norm{R_t^{(2)}}_\infty
        \le
        C\sum_{s=1}^{t-1} e^{-c(t-s)}\norm{\nu_{s+1}}_\infty .
\]
The Ces\`aro convergence in \eqref{eq:clt-hyp} and
Lemma~\ref{lem:cesaro-square-control} imply
\[
        \sup_{t\ge1}\frac1t\sum_{s=1}^t\norm{\nu_s}_\infty^2<\infty .
\]
Cauchy--Schwarz then gives
\[
\frac1t
\sum_{s=1}^{t-1}e^{-c(t-s)}\norm{\nu_{s+1}}_\infty
\le
\left(
\frac1t\sum_{s=1}^{t-1}e^{-c(t-s)}
\right)^{1/2}
\left(
\frac1t\sum_{s=1}^{t-1}e^{-c(t-s)}
\norm{\nu_{s+1}}_\infty^2
\right)^{1/2}
\to0 .
\]
Hence all three remainders, and also the terminal covariance
$\mathcal{E}_t\Sigma\mathcal{E}_t^\top$, vanish after division by $t$. It remains to identify $V_t$.  Since $F_{s+1}=\1\nu_{s+1}$,
\begin{equation}
        V_t
        =
        \sum_{s=1}^{t-1}
        \left(
        \nu_{s+1}\mathcal{E}_s\Sigma\mathcal{E}_s^\top\nu_{s+1}^\top
        \right)J .
\end{equation}
Let
\[
        \Lambda_s:=\mathcal{E}_s-\mathcal{E}_\infty .
\]
The elementary bound
\[
\left|
\nu_{s+1}\mathcal{E}_s\Sigma\mathcal{E}_s^\top\nu_{s+1}^\top
-
\nu_{s+1}\mathcal{E}_\infty\Sigma\mathcal{E}_\infty^\top\nu_{s+1}^\top
\right|
\le
C\norm{\Lambda_s}_\infty\norm{\nu_{s+1}}_\infty^2
\]
and the bounded Ces\`aro averages of $\norm{\nu_s}_\infty^2$ imply, by the usual
finite-initial-segment/tail split, that the Ces\`aro average of these differences converges to zero. The shifted and truncated Ces\`aro average has the same limit as the average in
the hypothesis.  Indeed, Lemma~\ref{lem:cesaro-square-control} gives
\[
        \norm{\nu_{t+1}}_\infty^2=o(t),
\]
and the fixed initial endpoint is negligible.  Thus the missing or added endpoint
terms are $o(t)$ after division by $t$.
Consequently,
\[
\frac1t\sum_{s=1}^{t-1}
\nu_{s+1}\mathcal{E}_\infty\Sigma\mathcal{E}_\infty^\top\nu_{s+1}^\top
=
\Tr\!\left[
\left(
\frac1t\sum_{s=1}^{t-1}\nu_{s+1}^\top\nu_{s+1}
\right)
\mathcal{E}_\infty\Sigma\mathcal{E}_\infty^\top
\right]\longrightarrow
\Tr\!\left(\mathcal{C}\mathcal{E}_\infty\Sigma\mathcal{E}_\infty^\top\right)
=\sigma^2 .
\]
Thus
\[
        \frac{\Gamma_t}{t}\to\sigma^2J .
\]
If the $\xi_s$ are Gaussian, then
$\frac{X_t-\Exp X_t}{\sqrt t}$
is centered Gaussian with covariance $\Gamma_t/t$, so the covariance convergence
proves the theorem in the Gaussian case.

For general i.i.d. perturbations, let $(Z_s)$ be i.i.d. centered Gaussian
vectors with covariance $\Sigma$, independent of $(\xi_s)$, and set
\[
        S_t:=\frac1{\sqrt t}\sum_{s=1}^tM_{t,s}\xi_s,
        \qquad
        T_t:=\frac1{\sqrt t}\sum_{s=1}^tM_{t,s}Z_s .
\]
Then $T_t\Rightarrow N(0,\sigma^2J)$ by the Gaussian case.  We show that $S_t$
and $T_t$ have the same limit.  For $s<t$ the rank-one estimate gives $\norm{M_{t,s}}_\infty
        \le
        C\bigl(\norm{\nu_{s+1}}_\infty+e^{-c(t-s)}\bigr)$, while for $s=t$ we only use the trivial bound
\[
        \norm{M_{t,t}}_\infty=\norm{\mathcal{E}_t}_\infty\le K .
\]
Therefore
\[
        \sup_t\frac1t\sum_{s=1}^t\norm{M_{t,s}}_\infty^2<\infty,
        \qquad
        q_t:=\max_{s\le t}\frac{\norm{M_{t,s}}_\infty}{\sqrt t}\to0 .
\]
The first bound follows from the bounded Ces\`aro averages of
$\norm{\nu_s}_\infty^2$.  For the second, Lemma~\ref{lem:cesaro-square-control}
gives $\max_{s\le t}\frac{\norm{\nu_{s+1}}_\infty}{\sqrt t}\to0$,  and the terminal term contributes only $K/\sqrt t$.

Let $\varphi\in C_b^3(\R^n)$ and define
\[
        R_\varphi(x,h)
        :=
        \varphi(x+h)-\varphi(x)-D\varphi(x)h
        -
        \tfrac12D^2\varphi(x)[h,h].
\]
Taylor's formula gives
\[
        |R_\varphi(x,h)|
        \le
        C_\varphi(\norm{h}_2^2\wedge\norm{h}_2^3),
\]
with $C_\varphi$ depending only on
$\norm{D^2\varphi}_\infty$ and $\norm{D^3\varphi}_\infty$.  For
$1\le s\le t$, put
\[
        U_{t,s}
        :=
        \frac1{\sqrt t}
        \left(
        \sum_{i<s}M_{t,i}Z_i+\sum_{i>s}M_{t,i}\xi_i
        \right).
\]
This vector is independent of both $\xi_s$ and $Z_s$.  The telescoping replacement identity gives
\[
\Exp\varphi(S_t)-\Exp\varphi(T_t)
=
\sum_{s=1}^t
\Exp\left[
\varphi\!\bigl(U_{t,s}+t^{-1/2}M_{t,s}\xi_s\bigr)
-
\varphi\!\bigl(U_{t,s}+t^{-1/2}M_{t,s}Z_s\bigr)
\right].
\]
Because $\xi_s$ and $Z_s$ have the same mean and covariance, the Taylor terms of
order $0,1,2$ cancel after conditioning on $U_{t,s}$.  Therefore, for each fixed
$\varepsilon>0$, splitting the expectations according to $\{\|\xi_s\|_2\ge \varepsilon/q_t\}$ and $\{\|Z_s\|_2\ge \varepsilon/q_t\}$, the coefficient bounds and the definition of $q_t$ yield
\[
\begin{aligned}
|\Exp\varphi(S_t)-\Exp\varphi(T_t)|
&\le
C\varepsilon
\frac1t\sum_{s=1}^t\norm{M_{t,s}}_\infty^2
\bigl(\Exp\norm{\xi_1}_2^2+\Exp\norm{Z_1}_2^2\bigr)
\\
&\quad+
C\frac1t\sum_{s=1}^t\norm{M_{t,s}}_\infty^2
\left(
\Exp\!\left[
\norm{\xi_1}_2^2
\mathbf1_{\{\norm{\xi_1}_2>\varepsilon/q_t\}}
\right]
+
\Exp\!\left[
\norm{Z_1}_2^2
\mathbf1_{\{\norm{Z_1}_2>\varepsilon/q_t\}}
\right]
\right).
\end{aligned}
\]
Letting $t\to\infty$ sends the second line to zero by dominated convergence and
the first coefficient bound.  Then $\varepsilon\downarrow0$ sends the first line
to zero.  Thus
\[
        \Exp\varphi(S_t)-\Exp\varphi(T_t)\to0
\]
for every $\varphi\in C_b^3(\R^n)$; a standard smoothing argument makes this
class convergence determining.  Hence
\[
        S_t\Rightarrow N(0,\sigma^2J).
\]
\emph{Second proof: scalar projection.}
Define
\[
        Y_t:=\nu_{t+1}X_t .
\]
By Proposition~\ref{prop:rankone}, $ \nu_s=\nu_{s+1}B_s$. Hence, from \eqref{eq:linear-recursion},
\[
        Y_t
        =
        \nu_{t+1}B_tX_{t-1}
        +
        \nu_{t+1}\mathcal{E}_t\gamma_t
        =
        \nu_tX_{t-1}
        +
        \nu_{t+1}\mathcal{E}_t\gamma_t .
\]
Iterating gives the exact scalar representation
\begin{equation}
        Y_t
        =
        \nu_1X_0
        +
        \sum_{s=1}^t\nu_{s+1}\mathcal{E}_s\gamma_s .
\end{equation}
Consequently,
\begin{equation}
        \frac{Y_t-\Exp Y_t}{\sqrt t}
        =
        \sum_{s=1}^t\eta_{t,s},
        \qquad
        \eta_{t,s}
        :=
        t^{-1/2}\nu_{s+1}\mathcal{E}_s\xi_s .
\end{equation}
The variance converges by the same shifted-Ces\`aro argument used in the first
proof:
\[
        \sum_{s=1}^t\operatorname{Var}(\eta_{t,s})
        =
        \frac1t\sum_{s=1}^t
        \nu_{s+1}\mathcal{E}_s\Sigma\mathcal{E}_s^\top\nu_{s+1}^\top
        \longrightarrow
        \Tr\!\left(\mathcal{C}\mathcal{E}_\infty\Sigma\mathcal{E}_\infty^\top\right)
        =
        \sigma^2 .
\]
For Lindeberg, set
\[
        a_{t,s}
        :=
        t^{-1/2}\mathcal{E}_s^\top\nu_{s+1}^\top
        \in\R^n,
        \qquad
        \eta_{t,s}=a_{t,s}^\top\xi_s .
\]
By Lemma~\ref{lem:cesaro-square-control},
\[
        \alpha_t
        :=
        \max_{s\le t}\norm{a_{t,s}}_2
        \le
        C t^{-1/2}\max_{s\le t}\norm{\nu_{s+1}}_\infty
        \to0 ,
\]
and the Ces\`aro convergence in \eqref{eq:clt-hyp} gives
$\sum_{s\le t}\norm{a_{t,s}}_2^2=O(1)$. Therefore, for every $\varepsilon>0$,
\[
        \sum_{s=1}^t
        \Exp\!\left[
        \eta_{t,s}^2\mathbf1_{\{|\eta_{t,s}|>\varepsilon\}}
        \right]
        \le
        \sum_{s=1}^t\norm{a_{t,s}}_2^2\,
        \Exp\!\left[
        \norm{\xi_1}_2^2
        \mathbf1_{\{\norm{\xi_1}_2>\varepsilon/\alpha_t\}}
        \right]
        \to0 .
\]
The Lindeberg--Feller theorem yields
\[
        \frac{Y_t-\Exp Y_t}{\sqrt t}
        \Rightarrow
        N(0,\sigma^2).
\]

It remains to lift the scalar limit back to $\R^n$.  Decompose
\[
        X_t=m_t\1+R_t,
        \qquad R_t\in H .
\]
Since $\nu_{t+1}\1=1$,
\[
        X_t-Y_t\1
        =
        R_t-(\nu_{t+1}R_t)\1 .
\]
By Lemma~\ref{lem:oscconv}, finite second moments of $\gamma_1$ imply
\[
        \sup_t\Exp\norm{R_t}_\infty<\infty .
\]
Moreover, Lemma~\ref{lem:cesaro-square-control} gives
\[
        \frac{\norm{\nu_{t+1}}_\infty}{\sqrt t}\to0 .
\]
Thus
\[
        \frac1{\sqrt t}
        \Exp\norm{X_t-Y_t\1}_\infty
        \le
        \frac{C}{\sqrt t}
        \bigl(1+n\norm{\nu_{t+1}}_\infty\bigr)
        \to0 .
\]
The same bound applies to the centered difference, with an extra harmless factor
$2$.  Hence
\[
        \frac{X_t-\Exp X_t}{\sqrt t}
        -
        \frac{Y_t-\Exp Y_t}{\sqrt t}\,\1
        \to0
        \quad\text{in }L^1\text{ and therefore in probability}.
\]
Slutsky's theorem gives \eqref{eq:clt-limit}.
\end{proof}

\begin{remark}
The first proof provides the full covariance matrix. The second proof shows the limit is a multiple of $J$ because Dobrushin contraction makes disagreement negligible, leaving only the scalar projection $\nu_{t+1}X_t$. The asymptotic variance aggregates rank-one terms through a Cesàro average, capturing the full trajectory of the limiting row weights $\nu_s$. Equivalently, the limiting Gaussian law resides on $\operatorname{Range}(F_t)=\mathbb{R}\1$.
\end{remark}


\begin{remark}[Pairwise Dobrushin geometry and dynamic agreement classes]
\label{rem:dynamic-classes}
The coefficient $\delta(B_{t:s})$ used above is obtained by taking the worst
row-pair distance.  If this maximum is not taken, one obtains a finer geometry
which can detect emergent agreement classes.  To do this, for $t\ge s$, set
\[
        Q_{t,s}:=B_{t:s},
        \qquad
        Q_{t,t+1}:=I,
\]
and define the pairwise Dobrushin distance
\[
        \Theta_{ij}(t,s)
        :=
        \frac12\sum_{k=1}^n
        \left|(Q_{t,s})_{ik}-(Q_{t,s})_{jk}\right|.
\]
Equivalently, since the rows of $Q_{t,s}$ have the same sum,
$$\Theta_{ij}(t,s) = \sup_{\osc(x)\le 1}\left|(Q_{t,s}x)_i-(Q_{t,s}x)_j\right|. 
$$
Thus $\delta(Q_{t,s})= \max_{1\le i,j\le n}\Theta_{ij}(t,s)$. For each initial time $s$, define
\[
        i\sim_s j
        \qquad\Longleftrightarrow\quad
        \Theta_{ij}(t,s)\longrightarrow0
        \quad\text{as }t\to\infty .
\]
The triangle inequality for the half-$\ell^1$ distance between rows shows
that $\sim_s$ is an equivalence relation.  Its equivalence classes are the
dynamically generated agreement classes: agents in the same class asymptotically
agree for every initial condition, while agents in different classes are not
forced to agree by the linear dynamics alone.  This is close in spirit to
cluster consensus and class-ergodicity for products of stochastic matrices; see
\cite{HanLuChen2013,TouriNedic2011,TouriNedicClasses}. The perturbed recursion satisfies the pairwise estimate
\[
        |X_t^i-X_t^j|
        \le
        \Theta_{ij}(t,1)\osc(X_0)
        +
        2K\sum_{\ell=1}^t
        \Theta_{ij}(t,\ell+1)\norm{\gamma_\ell}_\infty .
\]
Thus a pairwise analogue of \eqref{eq:mainhyp} is
\[
        A^*_{ij}
        :=
        \sup_{t\ge1}
        \sum_{\ell=1}^t
        \Theta_{ij}(t,\ell+1)
        <\infty .
\]
Under the global hypothesis \eqref{eq:mainhyp}, for every fixed initial time $s$,
\[
        \max_{i,j}\Theta_{ij}(t,s)
        \le
        \delta(B_{t:s})
        \le
        D_{t,s}
        \le
        Ce^{-c(t-s)} .
\]
Hence, for each fixed $s$, all pairwise row distances vanish exponentially as
$t\to\infty$, and the dynamically generated agreement relation has a single
class.  Proposition~\ref{prop:rankone} then gives the rank-one limit
$F_s=\1\nu_s$.  More generally, if the backward products converge to matrices
$F_s$ whose rows are equal within, but not necessarily across, the classes of
$\sim_s$, then the rank-one approximation $B_{t:s}\simeq \1\nu_s$ is replaced
by a block-row approximation $B_{t:s}\simeq F_s$.  In such a block-consensus
regime, the natural limiting fluctuations live in the associated block-consensus
subspace, rather than on the one-dimensional line $\operatorname{span}\{\1\}$.
The present paper treats the one-class case, for which $F_s=\1\nu_s$ and
$\operatorname{Range}(F_s)=\operatorname{span}\{\1\}$.
\end{remark} 



\begin{thebibliography}{99}

\bibitem{AAAP2025}
G. Amir, I. Arieli, G. Ashkenazi-Golan and R. Peretz.
Granular DeGroot dynamics -- a model for robust naive learning in social networks.
\emph{Journal of Economic Theory} \textbf{223} (2025), 105952.

\bibitem{ChatterjeeSeneta1977Consensus}
S. Chatterjee and E. Seneta.
Towards consensus: Some convergence theorems on repeated averaging.
\emph{Journal of Applied Probability} \textbf{14} (1977), 89--97.

\bibitem{degroot1974reaching}
M. H. DeGroot.
Reaching a consensus.
\emph{Journal of the American Statistical Association} \textbf{69} (1974),
118--121.

\bibitem{Dobrushin1956NonstationaryI}
R. L. Dobrushin.
Central limit theorem for nonstationary Markov chains. I.
\emph{Theory of Probability and its Applications} \textbf{1} (1956), 65--80.

\bibitem{Dobrushin1956NonstationaryII}
R. L. Dobrushin.
Central limit theorem for nonstationary Markov chains. II.
\emph{Theory of Probability and its Applications} \textbf{1} (1956), 329--383.

\bibitem{HanLuChen2013}
Y. Han, W. Lu and T. Chen.
Cluster consensus in discrete-time networks of multiagents with inter-cluster nonidentical inputs.
\emph{IEEE Transactions on Neural Networks and Learning Systems}
\textbf{24} (2013), 566--578.

\bibitem{popescu2023averaging}
I. Popescu and T. Vaidya.
Averaging plus learning models and their asymptotics. \emph{Proceedings of the Royal Society A: Mathematical, Physical and Engineering Sciences} \textbf{479} (2275) (2023), 20220681.

\bibitem{popescu2026mixednorms}
I. Popescu, J. Syatriadi and T. Vaidya. 
Anchoring and Mixed-Norm Contractions in Averaging-Learning Dynamics. arXiv preprint arXiv:2602.22627, (2026).

\bibitem{SethuramanVaradhan2005Dobrushin}
S. Sethuraman and S. R. S. Varadhan.
A martingale proof of Dobrushin's theorem for non-homogeneous Markov chains.
\emph{Electronic Journal of Probability} \textbf{10} (2005), 1221--1235.

\bibitem{Stenflo2008PerfectSampling}
\"O. Stenflo.
Perfect sampling from the limit of deterministic products of stochastic
matrices.
\emph{Electronic Communications in Probability} \textbf{13} (2008), 474--481.

\bibitem{TouriNedic2011}
B. Touri and A. Nedi\'c.
On ergodicity, infinite flow, and consensus in random models.
\emph{IEEE Transactions on Automatic Control} \textbf{56} (2011), 1593--1605.

\bibitem{TouriNedicClasses}
B. Touri and A. Nedi\'c.
On approximations and ergodicity classes in random chains.
\emph{IEEE Transactions on Automatic Control} \textbf{57} (2012), 2718--2730.



\end{thebibliography}
\end{document}